\documentclass[final]{siamart1116}

\usepackage{lipsum}
\usepackage{amsfonts}
\usepackage{graphicx}
\usepackage{epstopdf}
\usepackage{algorithmic}
\ifpdf
  \DeclareGraphicsExtensions{.eps,.pdf,.png,.jpg}
\else
  \DeclareGraphicsExtensions{.eps}
\fi

\numberwithin{theorem}{section}

\newcommand{\TheTitle}{On the Control of Density-Dependent Stochastic Population Processes with Time-Varying Behavior} 
\newcommand{\TheAuthors}{Y.~Lu, M.S.~Squillante, C.W.~Wu}

\headers{Control of Time-Varying Density-Dependent Population Processes}{\TheAuthors}

\title{{\TheTitle}\thanks{Submitted to the editors DATE.
\funding{This material is based upon work supported in part with funding from the Laboratory for Analytic Sciences (LAS). Any opinions, findings, conclusions, or recommendations expressed in this material are those of the author(s) and do not necessarily reflect the views of the LAS and/or any agency or entity of the United States Government.}}}

\author{
  Yingdong Lu\thanks{Mathematical Sciences Department, IBM T.J.\ Watson Research Center, Yorktown Heights, NY 10598, USA
	(\email{yingdong@us.ibm.com}).}
  \and
  Mark S.\ Squillante\thanks{Mathematical Sciences Department, IBM T.J.\ Watson Research Center, Yorktown Heights, NY 10598, USA
	(\email{mss@us.ibm.com}).}
  \and
  Chai Wah Wu\thanks{Mathematical Sciences Department, IBM T.J.\ Watson Research Center, Yorktown Heights, NY 10598, USA
	(\email{cwwu@us.ibm.com}).}
}

\usepackage{amsopn}

\ifpdf
\hypersetup{
  pdftitle={\TheTitle},
  pdfauthor={\TheAuthors}
}
\fi

\usepackage{dsfont}
\usepackage{mathrsfs}
\usepackage{times}
\usepackage{bm}
\usepackage{color}

\DeclareMathAlphabet{\mathpzc}{OT1}{pzc}{m}{it}
\DeclareMathOperator*{\argmax}{\arg\max}

\newcommand{\beqn}{\begin{equation}}
\newcommand{\eeqn}{\end{equation}}

\newcommand{\ex}{{\mathbb{E}}}
\newcommand{\pr}{{\mathbb{P}}}

\newcommand\Reals{{\mathbb{R}}}
\newcommand\Ints{{\mathbb{Z}}}

\newcommand{\g}{\lambda}                
\newcommand{\om}{\omega}                
\newcommand{\e}{\epsilon}                

\newcommand{\cC}{{\mathcal{C}}}
\newcommand{\cD}{{\mathcal{D}}}
\newcommand{\cF}{{\mathcal{F}}}

\newcommand{\cR}{{\mathcal{R}}}

\newcommand{\MSS}[1]{{\color{black}{ #1 }}}

\begin{document}

\maketitle

\begin{abstract}
The study of density-dependent stochastic population processes is important from a historical perspective
as well as from the perspective of a number of existing and emerging applications today.
In more recent applications of these processes, it can be especially important to include time-varying parameters
for the rates that impact the density-dependent population structures and behaviors.
 Under a mean-field scaling, we show that such
density-dependent stochastic population processes
with time-varying behavior converge to a corresponding dynamical system.
 We analogously establish that the optimal control of such
density-dependent stochastic population processes
converges to the optimal control of the limiting dynamical system.
 An analysis of both the dynamical system and its optimal control renders various important mathematical properties of interest.
\end{abstract}



\section{Introduction}
The general class of density-dependent stochastic population processes and the mathematical analysis of such processes
have a very rich and important history.
A starting point is likely the seminal work of Bernoulli on epidemiological models in the 1760s~\cite{Bern1766-simple,DieHee02}.
The general class of density-dependent population processes can be used to model any system that involves a population of similar
particles which interact,
such as processes with viral-propagation behaviors, logistic-growth behaviors, and chemical reaction behaviors~\cite[Chapter 11]{EthKur86}.
The study of these stochastic models continues to be important today across a wide variety of problem domains,
including a recent National Academy of Science report
on a land management program~\cite{NAS2013}.

Recent and emerging applications have received considerable attention in the research literature, which include mathematical models
of various aspects of large networks such as the complex structures and behaviors of communication networks, social media/networks,
viral-propagation networks (e.g., epidemics, computer viruses and worms), and financial networks;
refer to, e.g., \cite{GanMasTow05,EasKle10} and the references therein.
The study of social networks and related behaviors, in particular, continue to grow in importance and popularity;
see, e.g., \cite{BaTeBa16} and the references therein.
On the other hand, research on the control and optimization of these mathematical models of various aspects of large networks
has been much more limited; refer to, e.g., \cite{BoChGa+10}.
Even more importantly, this entire body of work has focused solely on static (non-time-varying) model parameters that impact
the complex structures and behaviors of the large networks of interest.

Our focus in this paper is on the general class of density-dependent stochastic population processes with time-varying
parameters.
Such time-varying behaviors often arise in many existing and emerging applications, especially those where one observes behaviors that
lead to forms of exacerbated complex dynamics and actions frequently found in communication, financial, social, and viral-propagation networks.
Our objective is twofold, namely to derive a mathematical analysis of such models and to derive the optimal control of these mathematical models.
In particular, we consider variants of the classical mathematical model of density-dependent stochastic population processes
analyzed by Kurtz \cite{Kurt71},\cite[Chapter 11]{EthKur86}, extending the analysis to first incorporate time-varying behavior
for the transition intensities of the Markov process and to then investigate aspects of the corresponding stochastic optimal
control problem.

We start by formally presenting a continuous-time, discrete-state density-dependent stochastic population process model in which
the state of each particle comprising the population and the dynamics of its state transitions are governed by functions of time.
Taking the limit as the population size tends to infinity under a mean-field scaling, we establish that this limiting stochastic
process converges in general to a continuous-state nonautonomous dynamical system.
In doing so, we generalize and extend the classical results of Kurtz~\cite{Kurt71},\cite[Chapter 11]{EthKur86} and the recent results
in \cite{ArmBec16,Armb16} to establish corresponding versions of these results that hold under time-varying parameters;
this involves technical arguments and details that are unique to the corresponding time-varying systems.
We then formally present a corresponding optimal control problem with respect to the controlled density-dependent stochastic population
process with time-varying parameters and establish an analogous result by showing that this optimally controlled stochastic process is
asymptotically equivalent to the optimal control of the limiting dynamical system as the population size tends to infinity under a mean-field scaling.
In doing so, we generalize and extend the results in \cite{GaGaBo12} to establish corresponding versions of these results that
hold under time-varying parameters;
once again, this involves technical arguments and details that are unique to the corresponding time-varying systems.

Our attention then turns to the limiting continuous-state nonautonomous dynamical system where we first derive various mathematical
properties of this system, including equilibrium points, asymptotic states, stability and related results.
It is well known that nonautonomous dynamical systems (e.g., $\dot{x} = f(x,t)$) can have vastly different and more complex behavior than
autonomous systems (e.g., $\dot{x} = f(x)$) even when the vector field $f$ is linear in $x$.
We then derive mathematical properties of the optimal dynamic control policy for the limiting continuous-state nonautonomous
dynamical system with the objective to maximize various instances of a general utility function.

It is important to note that our density-dependent stochastic population process model and results are quite general, and in particular
not at all restricted to the examples of viral propagation, logistic growth, and chemical reaction applications discussed herein.
More specifically, particles comprising the population can represent any entities of interest, the state of each particle can represent
any characteristics of interest, and the dynamics of state transitions can represent any phenomena of interest with respect to the particles
and their interactions.
In fact, our interest in these mathematical problems was motivated by a recent study of viral-propagation behaviors of people, energy sources,
and cybersystems~\cite{LuSqWu+15}.

The paper is organized as follows.
Section~\ref{sec:stochastic} presents our model and analysis of the general class of density-dependent stochastic population processes with time-varying parameters.
Section~\ref{sec:limit} presents our model and analysis of the limiting dynamical system, followed by concluding remarks.
Appendix~\ref{app:lemmas} contains some of our additional theoretical results and
Appendix~\ref{app:basicDST} contains some basic results from dynamical systems theory.

\section{Density-Dependent Stochastic Population Processes}
\label{sec:stochastic} 
We first define our model of the general class of density-dependent population processes with time-varying parameters and then
turn to establish that such a stochastic process is asymptotically equivalent to a set of ordinary differential equations (ODEs)
in the limit as the population size tends to infinity under a mean-field scaling.
We next show a similar result for the corresponding control problem by establishing that such an optimally controlled stochastic process
is asymptotically equivalent to the optimal control of the set of ODEs in the limit as the population size tends to infinity under a
mean-field scaling.
A special case of viral-propagation processes with time-varying parameters is then considered using an alternative set of arguments.

\subsection{Mathematical Model}
\label{sec:math-model}
Consider a sequence of Markov processes
$$\hat{Z}_n = \{ (\hat{X}_{n,1}(t), \ldots, \hat{X}_{n,d}(t)) ; t \geq 0 \}$$
indexed by the fixed parameter $n \in \Ints^+ := \{1, 2, \ldots \}$
and defined over the probability space $(\hat{\Omega}_n, \cF_n, \pr_n)$,
composed of the state space $\hat{\Omega}_n \subseteq \Ints^d$,
$\sigma$-algebra $\cF_n$ and probability measure $\pr_n$, with initial probability distribution $\bm\alpha_n$.
The fixed parameter $n$ has different interpretations depending upon the specific application and details of the
stochastic process of interest, but $n$ generically represents a form of the magnitude of a system involving similar particles that interact.
For example, in the context of logistic growth,
$n$ reflects the area of a region occupied by a certain population, $d=1$,
and the process $\hat{Z}_n(t)$ represents the population density at time $t$.
In the context of viral propagation,
$n$ reflects the total population size, $d=2$,
and the process $\hat{Z}_n(t)$ represents the ordered pair $(\hat{X}_n(t),\hat{Y}_n(t))$ of non-infected and infected
population at time $t$, respectively.
Lastly, in the context of chemical reactions,
$n$ reflects the volume of a chemical system containing $d$ chemical reactants,
and the process $\hat{Z}_n(t)$ represents the ordered tuple $(\hat{X}_1(t), \ldots, \hat{X}_d(t))$ of the numbers of
molecules of all reactants at time $t$.

Define $\Omega \subset \Reals^d$ and $\Omega_n := \Omega \cap \{ \ell/n : \ell \in \hat{\Omega}_n \}$.
The time-dependent infinitesimal generator $Q_n(t) = [q^{(n)}_{i,j}(t)]_{i,j\in\hat{\Omega}_n}$
for the Markov process $\hat{Z}_n$ has transition intensities that bear the general form
$q_{k, k+\ell}^{(n)}(t) = n \beta_{\ell, t}(k/n)$, for $k,k+\ell \in \hat{\Omega}_n$,
where $\beta_{\ell,t}(\cdot)$ are nonnegative functions defined on $\Omega$, for $\ell \in \hat{\Omega}_n$ and $t \geq 0$.
We assume throughout that $\beta_{\ell, t}(x)$ is continuous in $t$ and
that $(x + \ell/n) \in \Omega_n$ when $\beta_{\ell, t}(x) > 0$, both for $x \in \Omega_n$.
As a specific instance of this general form for logistic-growth processes,
in terms of the time-varying birth rate $\lambda(t)$ and death rate $\mu(t)$ proportional to the population size,
we consider the transition intensities
\begin{equation*}
q^{(n)}_{i,i+1}(t) = \lambda(t) \frac{i}{n} i = n\lambda(t) \frac{i}{n} \frac{i}{n} , \qquad
q^{(n)}_{i,i-1}(t) = \mu(t) \frac{i}{n} i = n\mu(t) \frac{i}{n} \frac{i}{n} ,
\end{equation*}
where the latter equalities are instances of the general form $n \beta_{\ell, t}(k/n)$.
For the specific instance of viral-propagation processes,
in terms of the time-varying infection rate $\lambda(t)$ and cure rate $\mu(t)$ proportional to fractions of the total population size,
we consider the transition intensities
\begin{equation}\label{eq:VP:q}
q^{(n)}_{(i,j),(i-1,j+1)}(t) = \lambda(t) i \frac{j}{n} = n\lambda(t) \frac{i}{n} \frac{j}{n} , \qquad
q^{(n)}_{(i,j),(i+1,j-1)}(t) = \mu(t) j = n\mu(t) \frac{j}{n} ,
\end{equation}
where the latter equalities are once again instances of the general form $n \beta_{\ell, t}(k/n)$.
The functions $\lambda(t)$ and $\mu(t)$ are assumed throughout to be continuous in $t$, consistent with the continuity assumption on $\beta_{\cdot, t}(\cdot)$.

We note that the above definition of the viral-propagation stochastic process $\hat{Z}_n$ is slightly different from the corresponding
(non-time-varying) model of Kurtz~\cite{Kurt71,EthKur86}, in that we allow an infected individual who is cured to become infected at a later time.
Both models assume connections among the population form a complete graph.
In any case, our results hold for both types of viral-propagation models as well as variations thereof with time-varying transition rates
$q_{k, k+\ell}^{(n)}(t)$ of the general form $n \beta_{\ell, t}(k/n)$.
Moreover, our results typically carryforward with little additional effort to an even more general form of
$q_{k, k+\ell}^{(n)}(t) = n ( \beta_{\ell, t}(k/n)+ O(1/n) )$~\cite[Chapter 11]{EthKur86}.

\subsection{Mean-Field Limit of Process}
We proceed by proving a stronger result that then implies the desired almost surely (a.s.) process limit for density-dependent population processes.
Suppose that the Markov Chain $\hat{Z}_n(t)$ is as defined above with time-dependent transition intensities of the general form
$q_{k, k+\ell}^{(n)}(t) = n \beta_{\ell, t}(k/n)$, for $k,k+\ell \in \hat{\Omega}_n$,
with nonnegative functions $\beta_{\ell,t}(x)$ defined as above on $\Omega$ for $\ell \in \hat{\Omega}_n$ and $t \geq 0$,
continuous in $t$, and Lipschitz continuous in $x=k/n$ (by definition), $x \in \Omega_n$.
Here we consider the parameter $n$ to be general, having different interpretations in different contexts.
From the martingale-problem method (see, e.g., \cite[Chapters~4,~6]{EthKur86}), we devise that $\hat{Z}_n(t)$ has the integral representation
\begin{align} \label{eqn:int_rep}
\hat{Z}_n (t) & = \hat{Z}_n(0) + \sum_\ell \ell W_\ell \left( n \int_0^t  \beta_{\ell, s}\left(\frac{\hat{Z}_n(s)}{n}\right) ds\right),
\end{align}
where the $W_\ell$ are independent standard Poisson processes.
Define $F_t (z) := \sum_\ell \ell \beta_{\ell, t}(z)$, $z \in \Omega_n$.
Further define $Z_n(t) := \hat{Z}_n(t)/n$ on the state space $\Omega_n$ with time-dependent transition intensities
$q_{i,j}^{(n)}(t) = n\beta_{n(j-i),t}(i)$, $i,j \in \Omega_n$.

Our strategy for the desired proof is to first obtain the integral representation of $Z_n(t)$, which leads to the generator of $Z_n(t)$
again through the martingale-problem method and the law of large numbers
for the Poisson process.
From this and the above we derive the desired expression
\begin{equation} \label{eqn:scaled_int_rep}
Z_n(t) \; = \; Z_n(0) + \sum_\ell  \frac{\ell}{n} \bar{W}_\ell \left( n \int_0^t \beta_{\ell, s} (Z_n(s)) ds \right) + \int_0^t F_s(Z_n(s))ds,
\end{equation}
where $\bar{W}_\ell$ denotes the centered Poisson process, i.e., $\bar{W}_\ell(x)= W_\ell (x)-x$.
It then follows, from known results for the time-dependent martingale problem (see, e.g., \cite[Chapter 7]{EthKur86}),
that the generator $A_n(t)$ for $Z_n(t)$ has the form
\begin{align} \label{eqn:gen_Z}
A_n(t) f(x) & = \sum_\ell n \beta_{\ell, t }(x) \bigg[f\Big(x+ \frac{\ell}{n}\Big) -f(x)\bigg] \nonumber \\
	& =  \sum_\ell n \beta_{\ell, t }(x) \bigg[f\Big(x+ \frac{\ell}{n}\Big) -f(x)- \frac{\ell \cdot \nabla f(x)}{n}\bigg]  + F_t(x) \cdot \nabla f(x),
\end{align}
for $x \in \Omega_n$.

One of our main results can now be presented, upon noting the following basic fact:
\begin{equation}
\label{eqn:centered}
\lim_{n\rightarrow \infty} \sup_{u\le v} \Big|\frac{\bar{W}_\ell(nu) }{n} \Big| =  0, \qquad a.s., v\ge 0 .
\end{equation}

\begin{theorem}
	Suppose that for each compact set $K \subset \Omega$
	\begin{equation*}
		\sum_\ell |\ell| \sup_{x\in K} \beta_{\ell,t} (x) < \infty, \qquad \forall t\ge 0 ,
	\end{equation*}
	and there exists $M_K>0$ such that
	\begin{equation}
	\label{eqn:lip}
	|F_t(x)-F_t(y) |\le M_K|x-y|, \qquad \forall x,y \in K, t \ge 0.
	\end{equation}
	Further supposing $Z_n(t)$ satisfies \eqref{eqn:scaled_int_rep}, $\lim_{n\rightarrow \infty} Z_n(0) = z_0$, and a process $Z(t)$ satisfies
	\begin{equation}
		Z(t) \; = \; z_0 + \int_0^t F_s(Z(s)) ds, \qquad t\ge 0,
	\label{eq:Z-dynamics}
	\end{equation}
	then we have, for every $t\ge 0$,
	\begin{equation}
	\label{eqn:main_conv}
	\lim_{n\rightarrow \infty}\sup_{s\le t} |Z_n(s) -Z(s)| =0, \qquad a.s.
	\end{equation}
	\label{thm:Kurtz-new2}
\end{theorem}
\begin{proof}
	We have
	\begin{align*}
		|Z_n(t)- Z(t) | & \le |Z_n(0) -z_0| + \Big|Z_n(t) -Z_n(0) -\int_0^tF_s(Z_n(s))ds\Big| \\
		&  + \Big|\int_0^t F_s(Z_n(s)) -F_s(Z(s))ds\Big|.
	\end{align*}
	From \eqref{eqn:lip}, we obtain
	\begin{align*}
		\Big|\int_0^t F_s(Z_n(s)) -  F_s(Z(s))ds\Big| 
		& \le \int_0^t |F_s(Z_n(s)) -F_s(Z(s))|ds \\
		& \le M \int_0^t  |Z_n(s) -Z(s)|ds .
	\end{align*}
	Define
	\begin{equation*}
		\epsilon_n(t) := \sup_{u\le t}  \Big|Z_n(u) -Z_n(0) -\int_0^uF_s(Z_n(s))ds\Big|,
	\end{equation*}
	which therefore yields
	\begin{equation*}
		|Z_n(t)- Z(t) | \le |Z_n(0) -z_0|+ \epsilon_n(t) + M \int_0^t  |Z_n(s) -Z(s)|ds .
	\end{equation*}
	Applying Gronwall's inequality then renders
	\begin{equation*}
		|Z_n(t)- Z(t) | \le (|Z_n(0) -z_0| + \epsilon_n(t))e^{Mt}.
	\end{equation*}
	Hence, we know that \eqref{eqn:main_conv} holds if $\lim_{n\rightarrow \infty} \epsilon_n(t) =0$.
	
	Meanwhile, from \eqref{eqn:scaled_int_rep}, we have
	\begin{equation*}
		\epsilon_n(t)  \le \sum_{\ell }  \frac{|\ell|}{n}\sup_{u\le t} | \bar{W}_\ell (n \bar{\beta}_{\ell,u}u) | ,
	\end{equation*}
	where $\bar{\beta}_{\ell, t} = \sup_{x\in \Omega_n} \beta_{\ell, t}(x)$.
	Furthermore, from the definition of $\bar{W}$, we obtain
	\begin{align*}
		\sup_{u\le t} | \bar{W} (n \bar{\beta}_{\ell,u}u)| &\le \sup_{u\le t}| W_\ell (n \bar{\beta}_{\ell,u}u) +  (n \bar{\beta}_{\ell,u}u)| \\
		& =  W_\ell (n \bar{\beta}_{\ell,t}t) +  (n \bar{\beta}_{\ell,t}t),
	\end{align*}
	where the equality is due to the monotonicity of the Poisson process.
	Hence,
	\begin{equation*}
		\epsilon_n(t) \le \sum_{\ell}\frac{|\ell|}{n} \big( W_\ell (n \bar{\beta}_{\ell,t}t) +  (n \bar{\beta}_{\ell,t}t) \big).
	\end{equation*}
	From the law of large numbers for the Poisson process, we can easily conclude that $\epsilon_n(t)$ is bounded by a constant.
	We then can apply the dominated convergence theorem, in conjunction with \eqref{eqn:centered}, to ensure that $\lim_{n\rightarrow \infty} \epsilon_n(t) =0$, a.s.
\end{proof}

From Theorem~\ref{thm:Kurtz-new2}, we then have that the stochastic process $Z_n(t)$ converges to a corresponding continuous-space
deterministic process $Z(t)$ a.s.\ as $n\rightarrow \infty$ and that $Z(t)$ satisfies a corresponding set of ODEs.
In particular, the process $Z(t)$ satisfies the integral form of the general nonautonomous dynamical system given in \eqref{eq:Z-dynamics} where
the specific details of the process and the corresponding set of ODEs depend upon $F_s(\cdot)$ for the original stochastic process $\hat{Z}_n(t)$.
As one such example, in the context of viral propagation, the stochastic process $Z_n(t)$ converges to a deterministic process $Z(t)=(X(t), Y(t))$
a.s.\ as $n\rightarrow \infty$ with $Z(t)$ satisfying the following pair of ODEs:
\begin{equation}
\frac{dX(t)}{dt}  = -\lambda(t) X(t) Y(t) + \mu(t) Y(t) , \qquad
\frac{dY(t)}{dt}  = \lambda(t) X(t) Y(t) - \mu(t) Y(t) .
\label{eq:dXY}
\end{equation}
This desired a.s.\ convergence result justifies the use of a continuous-state nonautonomous dynamical system to model a discrete-state real-world stochastic system.

\subsection{Mean-Field Analysis of Optimal Control}
We next turn our attention to an optimal control problem associated with the original general class of density-dependent stochastic population processes,
where our goal is to show that this control process is asymptotically equivalent to the optimal control of the corresponding set of ODEs as the population
size tends to infinity under a mean-field scaling.

Consider a sequence of controlled Markov processes $\hat{Z}_n(t)$, with the adaptive control process $u_n(t)$ that is realized with respect to the adaptive transition kernel $n \beta_{\ell,t}(k/n)$, $k,k+\ell \in \hat{\Omega}_n$,
recalling $\beta_{\ell,t}(\cdot)$ is continuous in $t$.
For each system indexed by $n$, the optimal control $u^*_n(t)$ is determined by solving the optimal control problem with respect to
the cost functions $c_1(\cdot)$ and $c_2(\cdot)$:
\begin{align*}
	\hat{J}^*_n(z) = &\min_{u_n(t)} \quad \hat{J}_n(z) \\
	= &\min_{u_n(t)} \quad \left\{ \int_0^T  c_1(\hat{Z}_n(t), u_n(t) ) dt + c_2(\hat{Z}_n(T))\right\}, \\
	&\mbox{ s.t. } \quad \hat{Z}_n(0) = z.  
\end{align*}
Here we assume the cost functions $c_1(z,u)$ and $c_2(z)$ are uniformly bounded, which is reasonable and justified by our interest in costs related only to the proportion of a population.
Recall the integral representation of $\hat{Z}_n(t)$ and $Z_n(t)$ in \eqref{eqn:int_rep} and \eqref{eqn:scaled_int_rep}, respectively.
Further recall that the generator $A_n(t)$ for $Z_n(t)$ has the form given in \eqref{eqn:gen_Z}.

For comparison towards our goal in this section, we also consider the corresponding optimal control problem associated with the limiting mean-field dynamical system of the previous section.
Namely, the optimal control $u^*(t)$ is determined by solving the corresponding optimal control problem with respect to the same cost functions $c_1(\cdot)$ and $c_2(\cdot)$,
which can be formulated as
\begin{align*}
	J^*(z) & = \min_{u_n(t)} \quad J(z) \\
	& = \min_{u_n(t)} \quad \left\{ \int_0^T  c_1(Z(t), u(t) ) dt + c_2(Z(T))\right\} , \\
	\mbox{ s.t. } & \quad Z(0) = z ,
\end{align*}
where $Z(t)$ follows the dynamics
\begin{align*}
	Z(t) = z + \int_0^t F_s(Z(s)) ds .
\end{align*}
Note that the function $F_s(\cdot)$ encodes the control information. 

We seek to show that the optimal control $u^*(t)$ in the limiting mean-field dynamical system provides an asymptotically equivalent optimal control $u^*_n(t)$
for the original system indexed by $n$ in the limit as $n$ tends toward infinity.
More specifically, we first establish the following main result.
\begin{theorem}\label{thm:control1}
	Let $\hat{Z}_n(t)$, $Z_n(t)$ and $Z(t)$ be as above.
	We then have
	\begin{align}
		\label{eqn:main}
		\lim_{n\rightarrow \infty} \hat{J}^*_n(z) = J^*(z).
	\end{align}
Furthermore, let $F_s^*(\cdot)$ denote the function that encodes the optimal control $u^*(t)$ of the limiting mean-field dynamical system.
Suppose the original stochastic process $\hat{Z}_n(t)$ follows the deterministic state-dependent control policy determined by $F_s^*(\cdot)$.
Then, asymptotically as $n\rightarrow\infty$ under a mean-field scaling, both systems will realize the same objective function value in \eqref{eqn:main}.
\end{theorem}
\begin{proof}
	We first want to show that
	\begin{align*}
		{\overline \lim}_{n\rightarrow \infty} \hat{J}^*_n(z) \; \le \; J^*(z) \; \le \; {\underline \lim}_{n\rightarrow \infty} \hat{J}^*_n(z).
	\end{align*}
	Given any $\e>0$, there exists an $F_s(z)$ such that $J(z) > J^*(z) -\e$ under $F_s(z)$ by definition.
	Now, consider a system indexed by $n$ that follows the deterministic policy determined by $F_s(z)$.
	From our mean field analysis in the previous section, we know
	\begin{align*}
		\lim_{n\rightarrow \infty} {\hat Z}_n(t) =Z(t), \qquad a.s.,
	\end{align*}
	since a fixed deterministic policy will be followed by both the system indexed by $n$ and the limiting dynamical system.
	In addition, because both $\hat{Z}(t)$ and $Z(t)$ are uniformly bounded and
	$c_1(\cdot)$ and $c_2(\cdot)$
	are uniformly bounded functions, we have
	\begin{align*}
		\lim_{n\rightarrow \infty} \hat{J}_n(z) = J(z) ,
	\end{align*}
	and therefore
	\begin{align*}
		{\underline \lim}_{n\rightarrow \infty} \hat{J}^*_n(z) \; \ge \; \lim_{n\rightarrow \infty} \hat{J}_n(z) \; = \; J(z) \; > \; J^*(z) -\e.
	\end{align*}
	
	Meanwhile, for each system indexed by $n$, we have a $Z_n$ under which $\hat{J}^*_n(z) \le \hat{J}_n(z)+\e$.
	Let $F_s(z)$, dependent on $n$, be as in the last term of \eqref{eqn:scaled_int_rep}.
	Given any sample path $\om$, define
	\begin{align}
		\label{eqn:aux}
		{\tilde Z}_n(t) := z + \int_0^t F_s({\tilde Z}_n(s))ds ,
	\end{align}
	where $F_s$ is different for different sample paths.
	Furthermore, define
	\begin{align*}
		{\tilde J}_n(z) :=  \ex\left[\int_0^T  c_1({\tilde Z}_n(t), u(t) ) dt + c_2({\tilde Z}_n(T))\right].
	\end{align*}
	We know that
	\begin{align*}
		{\overline \lim}_{n\rightarrow \infty} \tilde{J}_n(z) \; \le \; \hat{J}^*(z) .
	\end{align*}
	
	What remains is to determine an estimate of $|\hat{J}_n(z) - \tilde{J}_n(z)|$, for which we simply need to estimate
	\begin{align*}
		\ex[|\hat{Z}_n(t)-{\tilde Z}_n(t)|].
	\end{align*}
	From the martingale problem representation and equation \eqref{eqn:aux}, we can apply Gronwall's inequality and thus obtain
	\begin{align*}
		\ex[|\hat{Z}_n(t)-\tilde{Z}_n(t)|]& \le \ex\left[\sum_\ell  \frac{\ell}{n} \sup_{0\le s\le t}\bar{W}_\ell (A_n(s))\right]\exp[Bt]
	\end{align*}
	for some constant $B$. This implies that $\ex[|\hat{Z}_n(t)-\tilde{Z}_n(t)|]$ is a $O(1/n)$ term. Hence, we have
	\begin{align*}
		{\overline \lim}_{n\rightarrow \infty} \hat{J}_n(z) \; \le \; \hat{J}^*(z) +\e .
	\end{align*}
	The above arguments then lead to the desired result in \eqref{eqn:main}.

Finally, it is readily verified that the above result and arguments render the desired conclusion that the optimal control $u^*(t)$ in the limiting
mean-field dynamical system provides an asymptotically equivalent optimal control $u^*_n(t)$ for the original stochastic system indexed by $n$ in the
limit as $n \rightarrow \infty$.
\end{proof}

\subsection{Alternative Proof of Mean-Field Limit: Special Case}
We now revisit the special case of the viral-propagation processes of Section~\ref{sec:math-model},
in light of the recent alternative proof of the mean-field limit of such processes with fixed infection and cure rate parameters~\cite{ArmBec16,Armb16}.
Our goal is to generalize these results and extend these arguments to handle the case of time-varying infection and cure rate parameters,
where the technical details are unique to the corresponding time-varying systems.

Consider a sequence of Markov processes $$\hat{Z}_n = \{ (\hat{X}_n(t), \hat{Y}_n(t)) ; t \geq 0 \}$$ indexed by the total population size
$n \in \Ints^+$
and defined over the probability space $(\hat{\Omega}_n, \cF_n, \pr_n)$,
composed of the state space $$\hat{\Omega}_n := \{ (i,j) : 0 \leq i, j \leq n, i+j =n \},$$
$\sigma$-algebra $\cF_n$ and probability measure $\pr_n$, with initial probability distribution $\bm\alpha_n$.
Each process $\hat{Z}_n(t)$ represents the ordered pair $(\hat{X}_n(t),\hat{Y}_n(t))$ of non-infected and infected
population at time $t$, respectively, where we assume connections among the population form a complete graph.
The time-dependent infinitesimal generator $Q_n(t) = [q^{(n)}_{(i,j),(u,v)}(t)]$ for the Markov process $\hat{Z}_n$ 
has transition intensities given by~\eqref{eq:VP:q} in terms of the time-varying infection rate $\lambda(t)$
and cure rate $\mu(t)$, both of which are assumed throughout to be continuous in $t$.

Recalling the definition $Z_n(t) := \hat{Z}_n(t)/n$ over the state space
$$\Omega_n := \left\{ \left(\frac{i}{n},\frac{j}{n}\right) : 0 \leq i, j \leq n, i+j =n \right\},$$
we seek to show that the stochastic process $Z_n(t)$ converges to a deterministic process
$Z(t)=(X(t), Y(t))$ a.s.\ as $n\rightarrow \infty$ and that $Z(t)$ satisfies the pair of ODEs in~\eqref{eq:dXY}.
This desired a.s.\ convergence result is a process-level limit.
We view both the pre-limit and limit processes as elements of $D([0,\infty),[0,\infty))$, the space
of functions mapping from $[0,\infty)$ to $[0,\infty)$ that are right-continuous and have left limits (RCLL).
This space is endowed with the Skorohod $J_1$ topology~\cite{Whitt02}.
In particular, let $\Phi_m$ denote the class of strictly increasing, continuous mappings $\phi:[0,m]\rightarrow[0,m]$
such that $\phi(0)=0$ and $\phi(m)=m$. 
For $x,y\in D([0,\infty),[0,\infty))$, define
\begin{align*}
d_m(x,y) & :=\inf_{\phi\in\Phi_m}\{||\phi-e||_m\vee||x\circ\phi-y||_m\}, \\
d(x,y) & :=\sum_{m=1}^\infty 2^{-m}[d_m(x,y)\wedge 1] ,
\end{align*}
where $e(t)=t$ is the identity function.
Then the metric $d$ is the Skorohod $J_1$ metric in $D([0,\infty),[0,\infty))$.
Our convergence result states that $d(Z_n(\cdot),Z(\cdot))\rightarrow 0$ a.s.\ as $n\rightarrow\infty$.

The desired result for the above class of viral-propagation processes can be formally expressed by the following Theorem.
\begin{theorem}
The stochastic process $\hat{Z}_n(t)$ defined above converges a.s.\ as $n\rightarrow \infty$ to the deterministic process $Z(t)=(X(t), Y(t))$ such that
\begin{equation*}
\dot{X} = -\g(t) X(t) Y(t)+ \mu(t) Y(t) , \qquad \dot{Y} = \g(t) X(t) Y(t)-\mu(t) Y(t) .
\end{equation*}
Namely, $d(Z_n(\cdot),Z(\cdot))\rightarrow 0$ a.s.\ as $n\rightarrow\infty$.
\label{thm:Kurtz-new}
\end{theorem}
\begin{proof}
We proceed by focusing on the convergence of $Y_n(t):=\hat{Y}_n(t)/n$, which is sufficient to ensure the convergence of $Z_n(t)$ since $X_n(t)+ Y_n(t)=1$ with $X_n(t):=\hat{X}_n(t)/n$.
Suppose $Y_n(0)=y_0\in [0,1]$ for all $n$.
We show that, for any $T>0$,
\begin{equation}
\label{meansquareconv}
\lim_{n\rightarrow\infty}\sup_{t\in[0,T]}\ex[|Y_n(t)-Y(t)|^2]=0,
\end{equation}
where $Y(t)$ satisfies
\begin{equation}
Y'(t) = \g(t) [1-Y(t)]Y(t)-\mu(t) Y(t),\quad Y(0)=y_0.\label{ydefode}
\end{equation}

Our proof starts with
establishing an upper bound on $\ex[Y_n(t)]$, which is given in Lemma~\ref{lem:upper_bound} and makes use of Lemma~\ref{lem:uncond}, and
establishing a lower bound on $\ex[Y_n(t)]$, which is given in Lemma~\ref{lem:lower_bound}.
%
%
The next step is to show that the $z_n(t)$ process, defined by \eqref{ztdef} in Lemma~\ref{lem:lower_bound} together with $w_n(t)$ in \eqref{wtdef},
converges to the $Y(t)$ process uniformly in mean square as $n\rightarrow\infty$, in the sense of \eqref{meansquareconv}.
Consider a two-dimensional ODE system with a similar form as follows:
\begin{align}
\mathbf{v}(t):=(v_1(t),v_2(t)),\notag\\
v'_1(t)=\lambda(t)(v_1(t)-v_2(t))-\mu(t)v_1(t),\notag\\
v'_2(t)=2\lambda(t)(v_2(t)-v_2(t)^{1.5})-2\mu(t)v_2(t),\label{secondode}\\
v_1(0)=y_0, v_2(0)=y_0^2.\notag
\end{align}
Note that $\mathbf{v}(t)=(Y(t),Y(t)^2)$ is the unique solution to the above system of differential equations.
Moreover, as $n\rightarrow\infty$, the right hand side of \eqref{wtdef} converges to the right hand side of \eqref{secondode} if $\lim_{n\rightarrow\infty}w_n(t)=v_2(t)$.

Meanwhile, we know from \eqref{wtdef} that
\begin{align*}
w_n(t) = & y_0^2+ \int_0^t\g(s) [w_n(s) -w_n(s)^{1.5}] -2\mu(s) w_n(s) + \frac{1}{n} [\g(s)+\mu(s)]ds.
\end{align*}
We further know that the function $x-x^{1.5}$ has a maximum value of $4/9$ for $x\ge 0$, and $w_n(t) \ge 0$.
Hence, for a fixed $T$, we have, for any $t\in [0,T]$,
\begin{align*}
w_n(t) \le  y_0^2+ \int_0^t \Lambda_T \frac{4}{9} +\Lambda_T  + M_T ds = \bigg(\frac{13}{9} \Lambda_T+M_T\bigg) T
\end{align*}
where $\Lambda_T =\sup_{t\le T} \g(t) $ and $M_T= \sup_{t\le T} \mu(t)$.
This means that $|w_n(t)|$ is uniformly bounded.
In conjunction with \eqref{wtdef}, it follows that $\lim_{n\rightarrow\infty}w_n(t)=v_2(t)$.

Hence, $(z_n(t),w_n(t))$ converges to $(Y(t),Y(t)^2)$ uniformly on $[0,T]$ for any $T>0$.
\end{proof}

\section{Dynamical Systems}
\label{sec:limit} 
The limiting continuous-space deterministic process, as previously noted above, satisfies the integral form of the general nonautonomous dynamical system
in \eqref{eq:Z-dynamics}, where the specific details of the process and the corresponding set of ODEs depend upon $F_s(\cdot)$ for the original stochastic
process $\hat{Z}_n(t)$ and where the parameter $n$ has different interpretations depending upon such details of the original process.
We therefore primarily consider in this section one specific dynamical system, namely the deterministic process $z(t)=(x(t),y(t))$ resulting from Theorem~\ref{thm:Kurtz-new}.
At the end of this section, we discuss applications of our approach to address other types of dynamical systems.

\subsection{Model}
The results of Section~\ref{sec:stochastic} yield a corresponding continuous-time, continuous-state nonautonomous dynamical system
$(x(t),y(t))$, where $x(t)$ denotes the fraction of non-infected population
at time $t$ and $y(t)$ the fraction of infected population
at time $t$.
The starting state state $(x(0),y(0))$ of the system at time $t=0$ has initial probability distribution $\bm\alpha$.
Let $\lambda(t)$ denote the infection rate at time $t$ and $\mu(t)$ the cure rate at time $t$, for $t \in [0,T)$, where the
planning horizon $T$ can be finite or infinite.
We assume throughout that $\lambda(t), \mu(t) > 0$.

To elucidate the exposition, let us initially assume the infection rate $\lambda(t) > 0$ and cure rate $\mu(t) > 0$ are constant for all $t$;
namely, $\lambda(t) = \lambda$ and $\mu(t) = \mu$, $\forall t$.
The state equations are then given by:
\begin{equation*}
\frac{dx}{dt} = -\lambda x y + \mu y , \qquad\qquad
\frac{dy}{dt} = \lambda x y - \mu y ,
\end{equation*}
where $x$ and $y$ respectively describe the non-infected and infected population, 
with total population $c = x+y$.
Although our model definition implies $c=1$, we shall consider the case of general $c$ for mathematical completeness.

The dynamical system model defined above is continuously varying in time.
Within our mathematical framework, we also consider a more general model consisting of multiple regimes, each as defined above,
where there are jumps (positive or negative) in the state of the dynamical system and in the infection and cure rate functions
upon switching from one regime to another.
Assuming the length of each regime is sufficiently long to reach equilibrium before regime switching occurs
(a simple statement of differences in time-scale), without loss of generality, we can focus our mathematical analysis
on each regime in isolation where the equilibrium point for any regime becomes the starting point for the next regime.

Since $c = x+y$ and $\frac{d(x+y)}{dt} = 0$, we have $x(t)+y(t)= c = x(0)+y(0)$ for all $t$; i.e., the total population is constant.
Upon substituting $y = c-x$, we can equivalently rewrite the two-dimensional ODE as an one-dimensional ODE:
\begin{equation*}
\frac{dx}{dt} = \lambda x^2 - (\lambda c+\mu)x + \mu c .
\end{equation*}
We can then apply standard techniques to analyze this dynamical system and obtain the following result.
Note that the logistic growth model described in Section \ref{sec:math-model} and in \cite[Chapter 11]{EthKur86} also resulted in a one-dimensional ODE and amenable to a similar analysis.
\begin{theorem}
	For the dynamical system $(x(t),y(t))$ with $0 \leq x(0), y(0) \leq c$ and $x(t)+y(t)= c, \; \lambda(t) = \lambda, \;  \mu(t) = \mu$ for all $t$,
	the system has equilibrium points at $x_1^*=\frac{\mu}{\lambda}$ and $x_2^*=c$, and stability properties given by the three cases:
	\begin{enumerate}
		\item $\frac{\mu}{\lambda} < c$: The equilibrium point $x_1^*$ is stable and the equilibrium point $x_2^*$ is unstable.
		Moreover, all trajectories of the dynamical system will converge towards the equilibrium point $x_1^*$,
		with the sole exception of the initial state $x(0) = c$.
		\item $\frac{\mu}{\lambda} > c$: The equilibrium point $x_1^*$ is unstable and the equilibrium point $x_2^*$ is stable.
		Moreover, all trajectories of the dynamical system will converge towards the equilibrium point $x_2^*$.
		\item $\frac{\mu}{\lambda} = c$: There is one equilibrium point at $x_2^*$, which is neither stable nor unstable.
		Moreover, all trajectories of the dynamical system will converge towards the equilibrium point $x_2^*$.
	\end{enumerate}
	\label{thm:CTCS-fixed}
\end{theorem}
\begin{proof}
	First, we evaluate the derivative of $f$ at the two equilibrium points $x_1^*$ and $x_2^*$ to obtain
	\begin{align*}
	\frac{df(x)}{dx}|_{x=x^*_1} & = 2\lambda x^*_1-(\lambda c+\mu)=\mu-\lambda c , \\
	\frac{df(x)}{dx}|_{x=x^*_2} & = 2\lambda x^*_2-(\lambda c+\mu)=\lambda c - \mu .
	\end{align*}
	From the above equations for case $1$ and the Hartman-Grobman Theorem (Theorem~\ref{thm:Hartman-Grobman}),
	the equilibrium point $x_1^*$ is stable and the equilibrium point
	$x_2^*$ is unstable since $\frac{df(x)}{dx}|_{x=x^*_1} = \mu-\lambda c < 0$ and $\frac{df(x)}{dx}|_{x=x^*_2} = \lambda c -\mu > 0$.
	The convergence of all trajectories of the dynamical system then follows upon applying Lyapunov's second method for (global) stability
	(Theorem~\ref{thm:Lyapunov:global})
	together with the assumption $x(0), y(0) \geq 0$.
	
	Turning to the above equations under case $2$,
	the Hartman-Grobman Theorem (Theorem~\ref{thm:Hartman-Grobman}) renders that the equilibrium point $x_1^*$ is unstable and the
	equilibrium point $x_2^*$ is stable since $\frac{df(x)}{dx}|_{x=x^*_1} = \mu-\lambda c > 0$ and $\frac{df(x)}{dx}|_{x=x^*_2} = \lambda c -\mu < 0$.
	The convergence of all trajectories of the dynamical system then follows upon applying Lyapunov's second method for (global) stability
	(Theorem~\ref{thm:Lyapunov:global})
	together with the assumption $0 \leq x(0) \leq c$.
	
	Finally, from the above equations for case $3$ and the Hartman-Grobman Theorem (Theorem~\ref{thm:Hartman-Grobman}),
	there is one equilibrium point at $x_1^*=x_2^*=c$ that is neither
	stable nor unstable since $\frac{df(x)}{dx}|_{x=x^*_1} = \mu-\lambda c = 0$ and $\frac{df(x)}{dx}|_{x=x^*_2} = \lambda c -\mu = 0$.
	The convergence of all trajectories of the dynamical system then follows upon applying Lyapunov's second method for (global) stability
	(Theorem~\ref{thm:Lyapunov:global})
	together with the assumption $0 \leq x(0) \leq c$.
\end{proof}

To summarize, for the dynamical system of Theorem~\ref{thm:CTCS-fixed}, all trajectories will converge towards an equilibrium point,
which is at $x=\frac{\mu}{\lambda}$ when $\frac{\mu}{\lambda} < c$ and at $x=c$ when $\frac{\mu}{\lambda} \geq c$.
We now turn to the general instance of our dynamical system model with $\lambda(t)$ and $\mu(t)$ varying as functions of time $t$,
for which we have a more general result of a similar form.
\begin{theorem}
	For the dynamical system $(x(t),y(t))$ with $0 \leq x(0), y(0) \leq c$ and $x(t)+y(t)= c$, $\lambda(t)$, $\mu(t)$ continuously varying for all $t$,
	the system has an asymptotic state at $x_1^*(t)=\frac{\mu(t)}{\lambda(t)}$ and an equilibrium point at $x_2^*=c$, and stability properties given by the following four cases.
	\begin{enumerate}
		\item $0 < \frac{\mu(t)}{\lambda(t)} < \xi < c$, $\forall t$: The equilibrium point $x_2^*$ is unstable.
		Moreover, all trajectories of the dynamical system with initial state $x(0) < c$ will converge towards being
		eventually near the asymptotic state $x_1^*(t)$ with respect to a $\delta$-neighborhood, i.e.,
		$\|x(t) - \frac{\mu(t)}{\lambda(t)}\| \leq \delta$
		where $\delta$ is a nonnegative constant that depends on the rates of change of $\mu(t)$ and $\lambda(t)$.
		\item $\frac{\mu(t)}{\lambda(t)} > c$: The equilibrium point $x_2^*$ is stable.
		Moreover, all trajectories of the dynamical system will converge towards the equilibrium point $x_2^*$.
		\item $\frac{\mu(t)}{\lambda(t)} = c$: There is one equilibrium point at $x_2^*$, which is neither stable nor unstable.
		Moreover, all trajectories of the dynamical system will converge towards the equilibrium point $x_2^*$.
		\item $\frac{\mu(t)}{\lambda(t)} = 0$: There is one equilibrium point at $x_1^*=0$, which is neither stable nor unstable.
		Moreover, all trajectories of the dynamical system will converge towards this equilibrium point $x_1^*$.
	\end{enumerate}
	\label{thm:CTCS-vary}
\end{theorem}
\begin{proof}
	First note that if $x(0) < c$, then $x(t) < \psi$ for all $t$, for some $\psi< c$.  Next note that $\dot{x} = \lambda(t)\left(x-\frac{\mu(t)}{\lambda(t)}\right)\left(x-c\right)$.
	Consider the Lyapunov function $V(x,t) = \frac{1}{2}\left(x(t)-\frac{\mu(t)}{\lambda(t)}\right)^2$.
	The derivative of $V$ along trajectories is equal to 
	\begin{align*} \dot{V}(x) &= \frac{dV}{dx}\cdot \frac{dx}{dt} + \frac{dV}{dt} \\
	&= \left(x-\frac{\mu}{\lambda}\right)\left(\dot{x} - \frac{\mu'(t)\lambda(t)-\lambda'(t)\mu(t)}{\lambda^2(t)}\right) \\
	&= \lambda(t)\left(x-\frac{\mu}{\lambda}\right)^2\left(x-c\right)+ \left(x-\frac{\mu}{\lambda}\right) \left(\frac{\lambda'(t)\mu(t)-\mu'(t)\lambda(t)}{\lambda^2(t)}\right) .
	\end{align*}
	Note that $\dot{V} < 0$ if
	$\|x-\frac{\mu}{\lambda}\| > \left\|\frac{\lambda'(t)\mu(t)-\mu'(t)\lambda(t)}{\lambda^3(t)}\right\|/\|\psi-c\|$,
	and by setting $$\delta = \limsup_t  \left\|\frac{\lambda'(t)\mu(t)-\mu'(t)\lambda(t)}{\lambda^3(t)}\right\|/\|\psi-c\|, $$
	the result follows from a standard Lyapunov argument.
\end{proof}

To summarize, for the dynamical system of Theorem~\ref{thm:CTCS-vary}, all trajectories $x(t)$ will approach a $\delta$-neighborhood of
$\frac{\mu(t)}{\lambda(t)}$ when $0 < \frac{\mu(t)}{\lambda(t)} < c$, will approach $0$ when $\frac{\mu(t)}{\lambda(t)}=0$, and will approach $c$ when $\frac{\mu(t)}{\lambda(t)} \geq c$.

As a special case of Theorem~\ref{thm:CTCS-vary}, when $\lambda(t)$ and $\mu(t)$ asymptotically converge to a constant ratio,
then the equilibrium points and stability of such a continuously varying dynamical system are given by the following result.
\begin{theorem}
	For the dynamical system $(x(t),y(t))$ with $0 \leq x(0), y(0) \leq c$ and $x(t)+y(t)= c$, $\lambda(t)$, $\mu(t)$ continuously varying such that
	$\mu(t)/\lambda(t) \rightarrow \kappa$, the system has an asymptotic state at $x_1^*(t)=\frac{\mu(t)}{\lambda(t)}$ and an equilibrium point at $x_2^*=c$,
	and stability properties given by the following four cases.
	\begin{enumerate}
		\item $0 < \frac{\mu(t)}{\lambda(t)} < c$: The equilibrium point $x_2^*$ is unstable.
		Moreover, all trajectories of the dynamical system whose initial state is bounded away from $c$ will converge towards $x_1^*(t) \rightarrow \kappa$.
		\item $\frac{\mu(t)}{\lambda(t)} > c$:
		The equilibrium point $x_2^*$ is stable.
		Moreover, all trajectories of the dynamical system will converge towards the equilibrium point $x_2^*$.
		\item $\frac{\mu(t)}{\lambda(t)} = c$: There is one equilibrium point at $x_2^*$, which is neither stable nor unstable.
		Moreover, all trajectories of the dynamical system will converge towards the equilibrium point $x_2^*$.
		\item $\frac{\mu(t)}{\lambda(t)} = 0$: There is one equilibrium point at $x_1^*=0$, which is neither stable nor unstable.
		Moreover, all trajectories of the dynamical system will converge towards this equilibrium point $x_1^*$.
	\end{enumerate}
	\label{thm:CTCS-vary-special}
\end{theorem}

To illustrates the dynamics of the above mathematical results when the equilibrium state $\frac{\mu(t)}{\lambda(t)}$ converges to a constant (Theorem~\ref{thm:CTCS-vary-special}),
consider the system across different initial conditions $x(0) \in [0, c=1]$ according to a (truncated) normal distribution with mean 0.5;
refer to the two leftmost diagrams in Figure~\ref{fig:dynplot}.
The middle diagram in Figure~\ref{fig:dynplot} illustrates the trajectories of the system over time for the ten initial conditions $x(0) = 0.1, \; x(0) = 0.2, \; \ldots, \; x(0) = 1.0$; 
similarly, the diagram to its right illustrates the system trajectories over time for all initial conditions $x(0) \in [0,1]$ with the corresponding probability density function
color map from the leftmost diagram.
The rightmost diagram in Figure~\ref{fig:dynplot} illustrates the probability density function for the state of the system at the end of the time horizon.
Note that the closer the initial state $x(0)$ is to the unstable equilibrium point at $x=1$, the slower the trajectory converges to the equilibrium state.

\begin{figure*}[htbp]
	\centerline{\includegraphics[width=\textwidth]{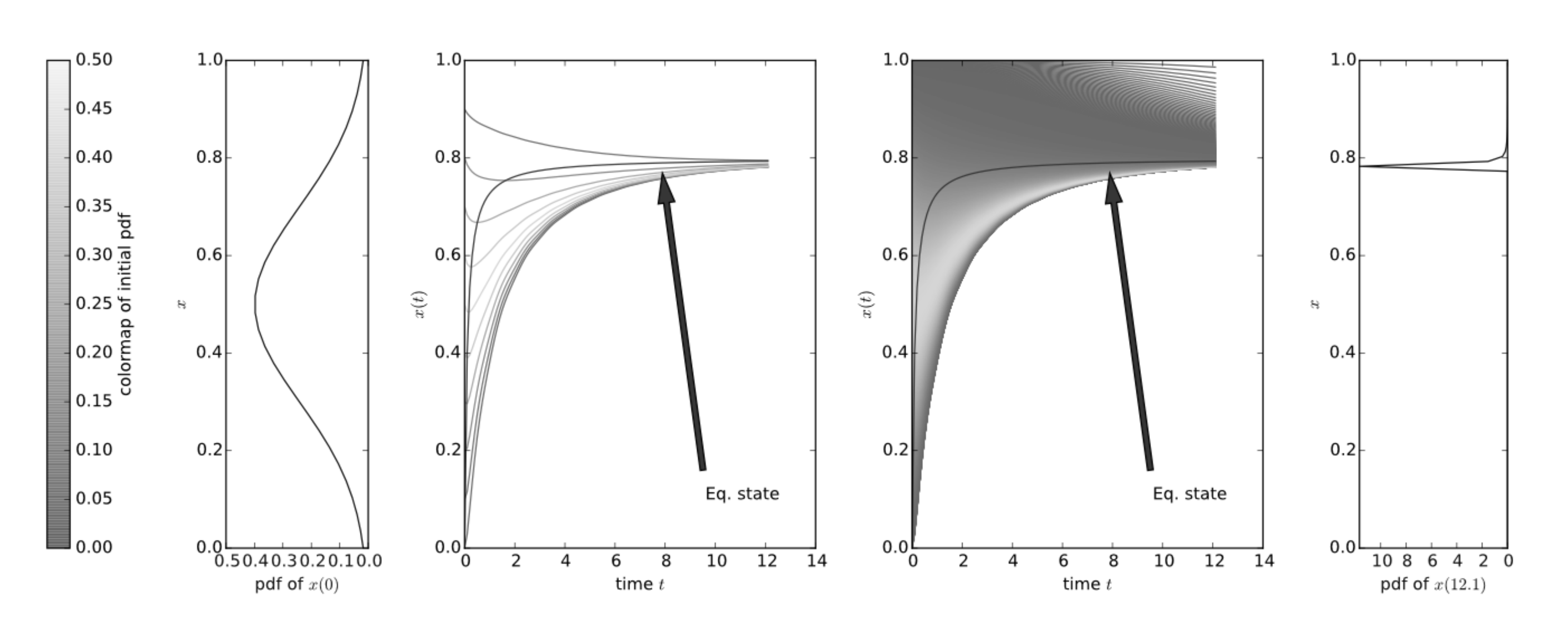}}
	\caption{Trajectories of $x(t)$.}\label{fig:dynplot}
\end{figure*}


\subsection{Optimal Control Results}
Consider the following optimal control formulation.
Let $R(x(t))$ and $C(y(t))$ denote the rewards and costs as a function of the state of the system at time $t$,
respectively.
More generally, we can have $R(\cdot,\cdot)$ and $C(\cdot,\cdot)$ each functions of both $x(t)$ and $y(t)$.
The decision variables are based on the controlled infection and cure rates $\lambda(t)$ and $\mu(t)$ deployed
by the system that represent changes from the original infection and cure rates, now denoted by $\hat{\lambda}(t)$
and $\hat{\mu}(t)$, where the system incurs costs $\hat{C}_\lambda(\cdot)$ and $\hat{C}_\mu(\cdot)$ as functions of
the deviations $\lambda(t) - \hat{\lambda}(t)$ and $\mu(t) - \hat{\mu}(t)$, respectively.
Throughout this subsection the control variables $\lambda(t)$ and $\mu(t)$ are assumed to be continuous in $t$,
with $\hat{\lambda}(t)$ and $\hat{\mu}(t)$ continuously varying for all $t$.
Define $\bm{\lambda} := (\lambda(t))$ and $\bm{\mu} := (\mu(t))$.
The objective function of our optimal control formulation is then given by
\begin{align}
\max_{\bm{\lambda}, \, \bm{\mu}} & \quad f\bigg( \, \int_{0}^T \Big\{ R(x(t)) - C(y(t)) -
\hat{C}_\lambda(\lambda(t) - \hat{\lambda}(t)) - \hat{C}_\mu(\mu(t) - \hat{\mu}(t))\Big\} dt \, \bigg), \label{opt:obj}
\end{align}
where $T$ denotes the time horizon, which can be finite or infinite, and $f(\cdot)$ represents an operator of
interest.
Let $\bm{\lambda}^*$ and $\bm{\mu}^*$ denote the optimal solution to \eqref{opt:obj} subject to the corresponding
ODEs of the previous section.

The above formulation represents the general case of the optimal control problem of interest.
Although there are no explicit solutions in general, this problem can be efficiently solved numerically using known methods from control theory.

To consider more tractable cases, and gain fundamental insights into the problem,
we start by first considering a one-sided version of this general problem in equilibrium with a fixed constant infection rate $\lambda=\hat{\lambda}=\hat{\lambda}(t)$
where the goal is to maximize the reward at the equilibrium point and only the parameter $\mu$ is under our control.
The optimal control in this case is a stationary policy for the cure rate, i.e., a single control $\mu$ in equilibrium.
Under a linear reward function with rate $\cR$ and linear cost functions with rates $\cC$ and $\hat{\cC}_\mu$,
we can rewrite the objective function \eqref{opt:obj} as
\begin{equation*}
\max_{\mu} \;\; \cR(x(\infty)) - \cC(y(\infty)) - \hat{\cC}_\mu(\mu) ,
\end{equation*}
since the optimal control is a stationary policy for the cure rate.
Upon substituting
$\min\{ c , \frac{\mu}{\lambda} \}$ for $x(\infty)$ and $c-x(\infty) = \max\{0,c-\frac{\mu}{\lambda}\}$
for $y(\infty)$, we derive the optimal control policy to be
\begin{align}
\mu^* = \argmax_{\mu \geq 0} \; \cR\bigg(\min\Big\{ c , \frac{\mu}{\lambda} \Big\} \bigg) - \cC\bigg(\Big[c-\frac{\mu}{\lambda}\Big]^+ \bigg) - \hat{\cC}_\mu(\mu) . \label{opt:one-sided:solution}
\end{align}
Namely, the optimal stationary control policy employs for all time $t$ the single control $\mu^*$ that solves \eqref{opt:one-sided:solution}.
An analogous formulation and result on $\lambda^*$ can be established for the opposite one-sided version of the problem in
equilibrium with constant cure rate $\mu$.

%
%



\MSS{
Next, as another step toward the general formulation,
consider the case where there are no costs for adjusting the infection and cure rates, i.e., $\hat{C}_\lambda(b) = 0 = \hat{C}_\mu(b)$ for all $b$.
Further assume that $(R(x)-C(y))$ has a single maximum at $(x^*,y^*)$,
which occurs when $R(\cdot)$ and $C(\cdot)$ are linear (in which case $x^*=0$ or $x^*=c=1$) or
when $R(\cdot)$ is concave and $C(\cdot)$ is convex (in which case $x^* \in [0,c=1]$).
We introduce the notion of an {\em ideal trajectory} denoted by $(x^I(t) = x^*, y^I(t)=y^*)$
that maximizes the objective function~(\ref{opt:obj}) at all time in this problem instance.
Hence, the optimal policy is to have $\frac{\mu(t)}{\lambda(t)} = x^*$ with $\lambda(t)$ as large as possible,
subject to $\frac{\hat{\mu}(t)}{\hat{\lambda}(t)}$ varying over time, since this governs the speed at which
$x(t)$ approaches and continually follows $x^*$.

More precisely, we establish a result showing that we can get arbitrarily close to the ideal trajectory, and thus the maximum objective.
Before doing so, we present the following related lemma on the general dynamics of the system.
\begin{lemma}
	For each $\epsilon > 0$ there is a $\hat{\delta} > 0$ such that if $0\leq x(0) < c-\epsilon$
	and $\lambda(t), \mu(t) > \hat{\delta}$ and $\frac{\mu(t)}{\lambda(t)} = x^*$ for all $t$, then
	$\left|x(t) - x^I(t)\right| < \epsilon$ for all $t$ sufficiently large.
	\label{lem:optimal}
\end{lemma}
\begin{proof}
	It is easy to show that $\frac{dx}{dt} \leq -a(x(t)-x^*)$ if $c-\epsilon > x(t) > x^*$
	and $\frac{dx}{dt} \geq a(x(t)-x^*)$ if $x(t) < x^*$ where $a =\epsilon\lambda(t)$.
	Hence we can make $a$ as large as possible by making $\lambda(t)$, and implicitly $\mu(t)$,
	as large as possible, and the conclusion follows from Theorem \ref{thm:comparison}.
\end{proof}

We can now present the main result of interest for this instance of the general formulation.
\begin{theorem}
	Suppose $\hat{C}_\lambda(b) = 0 = \hat{C}_\mu(b)$, for all $b$.
	For each $\epsilon > 0$ with $0\leq x(0) < c-\epsilon$, there is a $\hat{\delta} > 0$ such that if $\lambda(t), \mu(t) > \hat{\delta}$
	and $\frac{\mu(t)}{\lambda(t)} = x^*$ for all $t$, then the optimal solution of \eqref{opt:obj} is realized within $\epsilon$.
	\label{thm:optimal}
\end{theorem}
\begin{proof}
	The result directly follows as a consequence of Lemma~\ref{lem:optimal}, where we can continually make $\lambda(t)$, and implicitly $\mu(t)$,
	as large as possible to reach the optimal solution as fast as possible and to persistently follow the optimal solution as fast as possible.
\end{proof}

Let us next consider the above case where there are no costs for adjusting the infection and cure rates,
but where there are constraints on the rates of change of the control variables $\lambda(t)$ and $\mu(t)$,
i.e., $\theta_\lambda^\ell < \dot{\lambda} < \theta_\lambda^u$ and $\theta_\mu^\ell < \dot{\mu} < \theta_\mu^u$.
We continue to assume that $(R(x)-C(y))$ has a single maximum at $(x^*,y^*)$~~---~in which case $x^*=0$ or $x^*=c=1$ when $R(\cdot)$ and $C(\cdot)$ are linear;
or $x^* \in [0,c=1]$ when $R(\cdot)$ is concave and $C(\cdot)$ is convex.
Our above notion of an {\em ideal trajectory} remains the same,
namely $(x^I(t) = x^*, y^I(t)=y^*)$ maximizes the objective function~(\ref{opt:obj}) without constraints for all time $t$.
We therefore have that the optimal policy consists of setting $\lambda(t)$ and $\mu(t)$ so as to maximize the speed at which
$x(t)$ approaches and continually follows a maximum within an achievable neighborhood of $x^*$, subject to the constraints on
$\dot{\lambda}$ and $\dot{\mu}$ and subject to $\frac{\hat{\mu}(t)}{\hat{\lambda}(t)}$ varying over time.

More precisely, we establish a result showing that we can get arbitratily close to the best state within a $\delta$-neighborhood of the ideal trajectory,
and thus the maximum objective, where $\delta$ is a nonnegative constant that depends on the rates of change of $\hat{\lambda}(t)$ and $\hat{\mu}(t)$,
and on $\theta_\lambda^\ell , \theta_\lambda^u, \theta_\mu^\ell , \theta_\mu^u$.
Define $\cD(t) := \{ x(t) : \|x(t) - x^*\| \leq \delta \}$ for all $t$.
The main result of interest for this instance of the general formulation can then be expressed as follows.
\begin{theorem}
	Suppose $\hat{C}_\lambda(b) = 0 = \hat{C}_\mu(b)$, for all $b$, together with the constraints
	$\theta_\lambda^\ell < \dot{\lambda} < \theta_\lambda^u$ and $\theta_\mu^\ell < \dot{\mu} < \theta_\mu^u$.
	For each $\epsilon > 0$ with $0\leq x(0) < c-\epsilon$, there is a $\hat{\delta} > 0$ such that if $\lambda(t), \mu(t) > \hat{\delta}$
	and $\frac{\mu(t)}{\lambda(t)} = \hat{x}^*(t) := \argmax_{x(t) \in \cD(t)} (R(x(t))-C(y(t)))$ for all $t$, then the optimal solution of \eqref{opt:obj}
	under the constraints on $\dot{\lambda}$ and $\dot{\mu}$ is realized within $\epsilon$.
	\label{thm:optimal2}
\end{theorem}
\begin{proof}
	The result follows from the combination of arguments establishing Theorems~\ref{thm:optimal} and \ref{thm:CTCS-vary},
	where we can continually set $\lambda(t)$ and $\mu(t)$ so as to reach and persistently follow the best state $\hat{x}^*(t)$
	within a $\delta$-neighborhood of the optimal solution as fast as possible.
\end{proof}

When the costs for adjusting the infection and cure rates are introduced to either of the above instances of the general formulation,
the optimal policy will deviate from the ideal policies above where the deviation will depend on the initial state $x(0)$,
the cost functions $\hat{C}_\lambda(\cdot)$ and $\hat{C}_\mu(\cdot)$, the rates of change of $\hat{\lambda}(t)$ and $\hat{\mu}(t)$,
and any constraints on the rates of change of $\lambda(t)$ and $\mu(t)$.
Even though the policy of following the ideal trajectory is not optimal in general, it can provide structural properties and insight
into the complex dynamics of the system in a very simple and intuitive manner.
}

\subsection{Higher-Dimensional Dynamical Systems}
\label{sec:limit:general}
One of the benefits of reducing the asymptotic behavior of stochastic processes to a deterministic dynamical system is that the dynamical system can be more amenable to analysis,
especially when the system is autonomous.
Moreover, structural properties can be deduced by examining the state equations.
For instance it is well known that low dimensional systems cannot exhibit complex behavior.
In an autonomous dynamical system of the form $\dot{x} = f(x)$ where $f$ is continuous, oscillatory behavior is only possible if the dimension of $x$ is $2$ or higher;
and chaotic behavior is only possible if the dimension of $x$ is $3$ or higher~\cite{guckenheimer:1983}.
If $\frac{\partial f_i}{\partial x_j} \geq 0$ for all $i\neq j$
(respectively, if $\frac{\partial f_i}{\partial x_j} \leq 0$ for all $i\neq j$), then such systems are called cooperative
(respectively, competitive)\footnote{Both such systems were found to be useful in modeling various types of biological systems~\cite{smith:mds2012}.}
and in these cooperative systems there are no nontrivial periodic solutions that are attracting.
If in addition the Jacobian of $f$ is irreducible for all $x$, then almost every initial condition approaches the set of equilibrium points
and thus complex oscillatory behavior are not likely in such systems~\cite{hirsch:cooperative:1985}. 
For most density-dependent stochastic population processes,
the dynamics are bounded and hence the main dynamics are the trajectory approaching an equilibrium set.

When this is not the case and the dimensionality of the dynamical system is above $2$ or $3$, then the analysis of the dynamical system, as well as the original stochastic process,
is more complex.
Furthermore, when the system is nonautonomous as considered in this paper, the dynamics can be arbitrarily complex.
However, assuming the dynamical system parameters are varying at a much slower time scale than the dynamics and control of the system, then results in the analysis and control of slowly varying nonlinear dynamical systems can be brough to bear~\cite{peuteman:slowlyvarying:2002}.
At the same time, structural properties deduced from the state equations of the nonautonomous dynamical system and numerical simulation of these equations
render important characteristics and information about the asymptotic behavior and optimal control of the original stochastic process.

\MSS{
Theorem~\ref{thm:Kurtz-new2} shows that the stochastic process $Z_n(t)$ has mean-field behavior for large $n$ described by the integral form of the dynamical system in \eqref{eq:Z-dynamics}.
The reverse is also true: For every dynamical systems with a bounded invariant set, it is possible to construct a stochastic process whose mean-field behavior (as $n\rightarrow\infty$)
is described by the dynamics of the dynamical system.
There are many different stochastic processes whose asymptotic behavior maps to the same dynamical system.
One procedure for constructing such a stochastic process is roughly described as follows.
\begin{enumerate}
\item Shift the origin and rescale the state space such that the invariant set lies in $[0,B]^{d-1}$ and the vector field $\dot{x_i} = F_i(x,t)$, $1\leq i\leq d-1$, where $x = (x_1,\cdots, x_{d-1})$.
\item For each $i$, decompose the $i$-th component of the vector field $F_i(x,t)$ into $F_i(x,t) = P_i(x,t) - N_i(x,t)$, where $P_i(x,t)\geq 0$ and $N_i(x,t)\geq 0$.
\item Construct a stochastic process of $n$ agents and $d$ classes.
\item The number of agents in class $i$ is denoted $c_i$.
\item For $1\leq i \leq d-1$, the transition intensities of class $i$ to class $d$ are given by
\begin{eqnarray*}
q^{(n)}_{c_i \rightarrow c_i+1, c_{d}\rightarrow c_{d}-1} & = & \alpha P_i(x, t),\\
q^{(n)}_{c_i \rightarrow c_i-1, c_{d}\rightarrow c_{d}+1} & = & \alpha N_i(x, t),
\end{eqnarray*}
where $x_i = \frac{c_i}{n}$ and $\alpha > 0$ is some fixed constant.
\end{enumerate}

If the decomposition of the vector field into $P_i$ and $N_i$ is not easily obtained, an alternative procedure for constructing such a stochastic process is as follows.
\begin{enumerate}
\item Shift the origin and rescale the state space such that the invariant set lies in $[0,B]^{d-1}$ and the vector field $\dot{x_i} = F_i(x,t)$, $1\leq i\leq d-1$.
\item Construct a stochastic process of $n$ agents and $d$ classes.
\item The number of agents in class $i$ is denoted $c_i$.
\item For $1\leq i \leq d-1$, the transition intensities of class $i$ to class $d$ are given by
$$q^{(n)}_{c_i \rightarrow c_i+\mbox{sgn}(a), c_{d}\rightarrow c_{d}-\mbox{sgn}(a)} = |a|,$$
where $a =  \alpha F_i(x, t)$ and $x_i = \frac{c_i}{n}$ and $\alpha > 0$ is some fixed constant.
\end{enumerate}

As one specific example, along the lines of a $d$-dimensional viral propagation process,
applying the first procedure to the well-known Lorenz system~\cite{lorenz:1963} (which admits a decomposition into $P_i$ and $N_i$)
yields a stochastic process with $d=4$ classes and transitition intensities described by
\begin{equation*}
\begin{array}{lcl}
q^{(n)}_{c_1 \rightarrow c_1+1, c_{4}\rightarrow c_{4}-1} &=& \alpha a \left(x_1+\frac{1}{4}\right),\\
q^{(n)}_{c_1 \rightarrow c_1-1, c_{4}\rightarrow c_{4}+1} &=& \alpha \left(\frac{100x_1x_3}{3}+x_2+\frac{b}{3}\right),\\
q^{(n)}_{c_2 \rightarrow c_2+1, c_{4}\rightarrow c_{4}-1} &=& \alpha (cx_3+24(x_1+x_2)),\\
q^{(n)}_{c_2 \rightarrow c_2-1, c_{4}\rightarrow c_{4}+1} &=& \alpha  \frac{3ax_2}{2},\\
q^{(n)}_{c_3 \rightarrow c_3+1, c_{4}\rightarrow c_{4}-1} &=& \alpha \left(\frac{2bx_1}{3}+\frac{50x_3}{3}+\frac{1}{2}\right),\\
q^{(n)}_{c_3 \rightarrow c_3-1, c_{4}\rightarrow c_{4}+1} &=& \alpha (48x_1x_2+12).
\end{array}
\end{equation*}

We use the parameters $a=10$, $b=28$ and $c=\frac{8}{3}$, which are the standard parameters for the Lorenz system to produce the butterfly chaotic attractor.
Simulating this stochastic process with $\alpha = 0.015$ and $n =6000$ for $5000000$ iterations renders the values of $x_i$ whose phase portrait and time series
are illustrated in Figures \ref{fig:lorenz} and \ref{fig:lorenzts}, respectively.
The value of $x_4$ is not shown since $x_4 = 1-(x_1+x_2+x_3)$ can be derived from the other components.
These figures clearly show that the output of the stochastic process shares the features of the Lorenz chaotic attractor, even for a relatively small value of $n$.

\begin{figure}[htbp]
\includegraphics[width=5in]{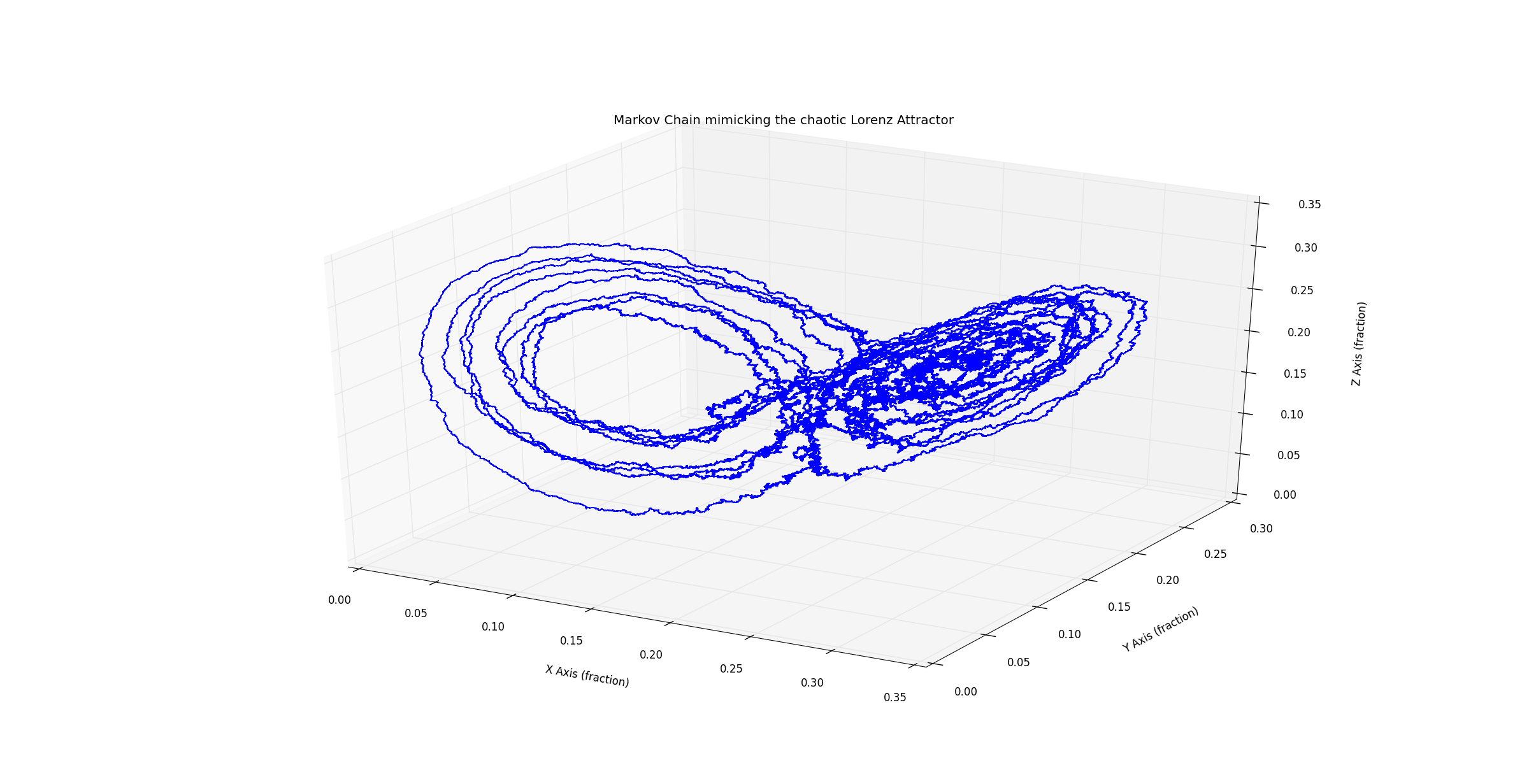}
\caption{Phase portrait from the stochastic process mimicking a Lorenz attractor}
\label{fig:lorenz}
\end{figure}

\begin{figure}[htbp]
\includegraphics[width=5in]{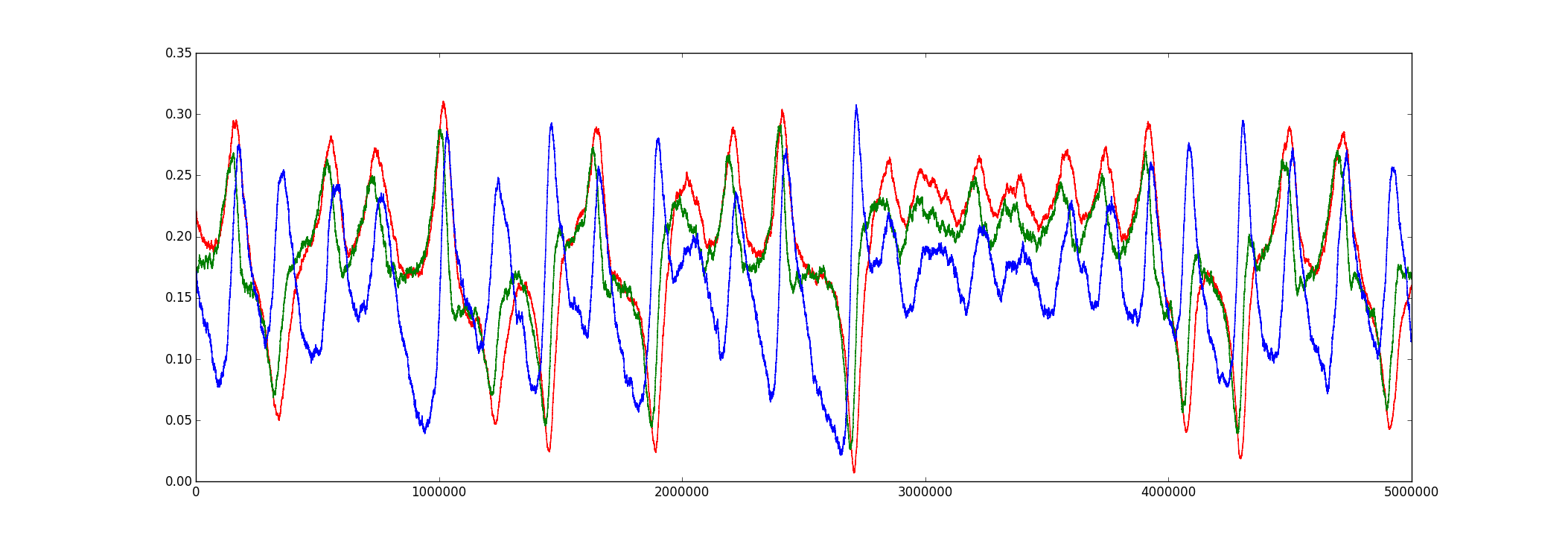}
\caption{Time series from the stochastic process mimicking a Lorenz attractor}
\label{fig:lorenzts}
\end{figure}
}

\section{Conclusion}
Motivated by current and emerging applications of today, we considered in this paper the general class of density-dependent stochastic population processes
with time-varying behavior.
We have established that this class of stochastic processes, under a mean-field scaling, converges to a corresponding class of nonautonomous dynamical systems,
thus extending classical results for such density-dependent population processes  without time-varying behavior.
A special case of viral-propagation processes is considered, thus extending recent results of mean-field limits for such processes to support time-varying parameters.
We also analogously show that the optimal control of the general class of density-dependent stochastic population processes converges to the optimal control of
the corresponding class of limiting nonautonomous dynamical systems.
Important mathematical properties of interest are derived through an analysis of the dynamical system and its optimal control.

\appendix
\section{}
\label{app:lemmas}
This appendix presents a few Lemmas, providing upper and lower bounds on $\ex[Y_n(t)]$, that are used in the proofs of some of our main results.

\begin{lemma}
	\label{lem:uncond}
\begin{equation}\label{uncond}
\frac{\partial \ex[\hat{Y}_n(t)]}{\partial t}=\frac{\lambda(t)}{n}\ex[(n-\hat{Y}_n(t))\hat{Y}_n(t)]-\mu(t)\ex[\hat{Y}_n(t)].
\end{equation}
\end{lemma}
\begin{proof}
Note that
\begin{equation} \label{cond}
\lim_{h\rightarrow 0}\frac{\ex[\hat{Y}_n(t+h)-\hat{Y}_n(t)|\hat{Y}_n(t)\in (0,n)]}{h} = \frac{\lambda(t)}{n}[n-\hat{Y}_n(t)]\hat{Y}_n(t)-\mu(t)\hat{Y}_n(t)
\end{equation}
because, for small $h$, $\hat{Y}_n(t+h)-\hat{Y}_n(t)$ is equal to $1$ and $-1$ with probability $\frac{\lambda(t)}{n}[n-\hat{Y}_n(t)]\hat{Y}_n(t)h$ and $\mu(t)\hat{Y}_n(t)h$,
respectively, and takes on all other values with probability $o(h)$.
Taking the expectation of \eqref{cond} and further interchanging the differentiation and expectation operators,
which is allowed since $\hat{Y}_n(t)$ takes on only finitely many possible values for a fixed $n$, leads to \eqref{uncond}.
\end{proof}

\begin{lemma}[Upper Bound]
	\label{lem:upper_bound}
$\ex[Y_n(t)] \leq Y(t)$. \\
\end{lemma}
\begin{proof}
From Lemma~\ref{lem:uncond}, we divide both sides of \eqref{uncond} by $n$ to obtain
\begin{equation*}
\frac{\partial \ex[Y_n(t)]}{\partial t}=\lambda(t)\ex\Big[\frac{n-\hat{Y}_n(t)}{n}\cdot\frac{\hat{Y}_n(t)}{n}\Big]-\mu(t)\ex\left[\frac{\hat{Y}_n(t)}{n}\right],
\end{equation*}
or
\begin{equation}
\frac{\partial \ex[Y_n(t)]}{\partial t}=\lambda(t)\left(\ex[Y_n(t)]-\ex[Y_n(t)^2]\right)-\mu(t)\ex[Y_n(t)].  \label{ubeq}
\end{equation}
Applying $\ex[Y_n(t)^2]\geq \ex[Y_n(t)]^2$ to \eqref{ubeq} yields
\begin{align*}
\frac{\partial \ex[Y_n(t)]}{\partial t}&\leq\lambda(t)\left(\ex[Y_n(t)]-\ex[Y_n(t)]^2\right)-\mu(t)\ex[Y_n(t)]\notag\\
&=\lambda(t)\left(1-\ex[Y_n(t)]\right)\ex[Y_n(t)]-\mu(t)\ex[Y_n(t)] .
\end{align*}
Upon combining
this
and the definition of $Y(t)$ (i.e., ODE \eqref{ydefode}), we have $\ex[Y_n(t)]\leq Y(t)$ for all $t$ due to Theorem~\ref{thm:comparison}.
\end{proof}

\begin{lemma}[Lower Bound]
	\label{lem:lower_bound}
Define a function $z_n(t)$ such that $z_n(0)=y_0$ and
\begin{equation}\label{ztdef}
z'_n(t)=\lambda(t)\left(z_n(t)-w_n(t)\right)-\mu(t)z_n(t),
\end{equation}
where $w_n(t)$ satisfies $w_n(0)=y^2_0$ and
\begin{align}\label{wtdef}
	w'_n(t)=2\lambda(t)(w_n(t)-w_n(t)^{1.5})-2\mu(t)w_n(t) +\frac{1}{n}\left[\lambda(t)+\mu(t)\right].
\end{align}
Then we have $z_n(t)\leq \ex[Y_n(t)]$ for all $t$.
\end{lemma}
\begin{proof}
Similar to the argument in the proof of the upper bound, since for small $h$, $\hat{Y}_n(t+h)^2-\hat{Y}_n(t)^2$
is equal to $(\hat{Y}_n(t)+1)^2-\hat{Y}_n(t)^2$ and $(\hat{Y}_n(t)-1)^2-\hat{Y}_n(t)^2$ with probability
$\frac{\lambda(t)}{n}[n-\hat{Y}_n(t)]\hat{Y}_n(t)h$ and $\mu(t)\hat{Y}_n(t)h$, respectively,
and taking on all other values with probability $o(h)$, we have
\begin{align*}
\frac{\partial \ex[\hat{Y}_n(t)^2]}{\partial t}\notag
&=\ex\left[((\hat{Y}_n(t)+1)^2-\hat{Y}_n(t)^2)\cdot\frac{\lambda(t)}{n}\hat{Y}_n(t)(n-\hat{Y}_n(t))\right]\notag\\
&\quad +\ex[((\hat{Y}_n(t)-1)^2-\hat{Y}_n(t)^2)\cdot\mu(t)\hat{Y}_n(t)].
\end{align*}
Upon dividing by $n^2$ on both sides together with simple term rearrangements, this becomes
\begin{align}
\frac{\partial \ex[Y_n(t)^2]}{\partial t}\notag
=&\lambda(t)\ex\left[\frac{(\hat{Y}_n(t)+1)^2-\hat{Y}_n(t)^2}{n}\cdot Y_n(t)(1-Y_n(t))\right]\notag\\
&\quad +\mu(t)\ex\left[\frac{((\hat{Y}_n(t)-1)^2-\hat{Y}_n(t)^2)}{n}\cdot Y_n(t)\right] \nonumber\\
=&\lambda(t)\ex\left[\frac{2\hat{Y}_n(t)+1}{n}\cdot Y_n(t)(1-Y_n(t))\right]
+\mu(t)\ex\left[\frac{-2\hat{Y}_n(t)+1}{n}\cdot Y_n(t)\right] \nonumber\\
=&\lambda(t)\ex\left[(2Y_n(t)+1/n)\cdot (Y_n(t)-Y_n(t)^2)\right]
+\mu(t)\ex\left[(-2Y_n(t)+1/n)\cdot Y_n(t)\right] \nonumber\\
=&\lambda(t)\ex\left[2Y_n(t)^2-2Y_n(t)^3+Y_n(t)/n-Y_n(t)^2/n\right]
+\mu(t)\ex\left[-2Y_n(t)^2+Y_n(t)/n\right] \nonumber\\
=&2\lambda(t)(\ex[Y_n(t)^2]-\ex[Y_n(t)^3])-2\mu(t)\ex[Y_n(t)^2] \notag\\
&\quad +\frac{1}{n}\Big[\lambda(t) (\ex[Y_n(t)]-\ex[Y_n(t)^2])+\mu(t)\ex[Y_n(t)]\Big] . \label{interexp}
\end{align}

We next apply $\ex[Y_n(t)^3]\geq \ex[Y_n(t)^2]^{1.5}$ to the $-\ex[Y_n(t)^3]$ inside the parentheses of the first term of \eqref{interexp},
and $\ex[Y_n(t)^2]\geq \ex[Y_n(t)]^2$ to the $-\ex[Y_n(t)^2]$ inside the parentheses on the second line of \eqref{interexp},
and thus obtain
\begin{align*}
\frac{\partial \ex[Y_n(t)^2]}{\partial t}
\leq&2\lambda(t)(\ex[Y_n(t)^2]-\ex[Y_n(t)^2]^{1.5})-2\mu(t)\ex[Y_n(t)^2]\notag\\
&\quad +\frac{1}{n}\left[\lambda(t) (\ex[Y_n(t)]-\ex[Y_n(t)]^2)+\mu(t)\ex[Y_n(t)]\right] .
\end{align*}
Noting that $\ex[Y_n(t)]$ and $1-\ex[Y_n(t)]$ are both within [0,1], we have
\begin{align}
\frac{\partial \ex[Y_n(t)^2]}{\partial t}
\leq&2\lambda(t)(\ex[Y_n(t)^2]-\ex[Y_n(t)^2]^{1.5})-2\mu(t)\ex[Y_n(t)^2]
+\frac{1}{n}\left[\lambda(t)+\mu(t)\right] . \label{ysquareine}
\end{align}

Now applying Theorem~\ref{thm:comparison} to $\ex[Y_n(t)^2]$ with respect to \eqref{ysquareine}, we find that
\begin{equation}
\ex[Y_n(t)^2]\leq w_n(t),\label{ubonsquare}
\end{equation}
by the defition of $w_n(t)$. Substituting \eqref{ubonsquare} into \eqref{ubeq} then leads to 
\begin{equation}\label{yine}
\frac{\partial \ex[Y_n(t)]}{\partial t}\geq\lambda(t)\left(\ex[Y_n(t)]-w_n(t)\right)-\mu(t)\ex[Y_n(t)] .
\end{equation}
Upon treating $w_n(t)$ as an exogenous function and applying Theorem~\ref{thm:comparison} once again, this time to $\ex[Y_n(t)]$ with respect to \eqref{yine},
we obtain the desired lower bound
\begin{equation*}
\ex[Y_n(t)]\geq z_n(t),~t\geq 0,
\end{equation*}
following \eqref{ztdef}.
\end{proof}

\section{} 
\label{app:basicDST}
This appendix provides some basic and classical results from dynamical systems theory that are exploited to establish some of our main results.

\begin{theorem}[Hartman-Grobman~\cite{Grob59,Hart60,Grob62,Hart63}]
	If a $d^{th}$-order system of differential equations has an equilibrium $v$ with linearization matrix $A$, and if $A$ has no zero or pure imaginary
	eigenvalues, then the phase portrait for the system near the equilibrium is obtained from the phase portrait of the linearized system $Dx = Ax$
	via a continuous change of coordinates.
	\label{thm:Hartman-Grobman}
\end{theorem}


\begin{theorem}[Lyapunov Global Stability~\cite{Lyap1892,LaSLef61,Vidy02}]
	Assume that there exists a scalar function $V$ of the state $x$, with continuous first order derivatives such that
	\begin{itemize}
		\item $V(x)$ is positive definite ,
		\item $\dot{V}(x)$ is negative definite ,
		\item $V(x) \rightarrow \infty$ as $\|x\| \rightarrow \infty$ ,
	\end{itemize}
	then the equilibrium at the origin is globally asymptotically stable.
	\label{thm:Lyapunov:global}
\end{theorem}

The following result is a consequence of Gronwall's inequality:
\begin{theorem}[Comparison theorem for scalar differential equations~\cite{blah1,blah2}]
	Suppose $f(x,t)$ is continuous in $t$ and Lipschitz continuous in $x$.   Suppose $u(t)$ and $v(t)$ are $C^1$ functions such that
	$\frac{du}{dt} \leq f(u(t),t)$ and $\frac{dv}{dt} = f(v(t),t)$.
	If $u(t_0) \leq v(t_0)$, then $u(t)\leq v(t)$ for $t\geq t_0$.
	\label{thm:comparison}
\end{theorem}

%

\section*{Acknowledgment}
The authors would like to acknowledge and thank B.~Zhang of IBM Research for
his important contributions to the proof of Theorem~\ref{thm:Kurtz-new}.



\bibliographystyle{siamplain}
\bibliography{main}

\begin{thebibliography}{10}

\bibitem{Armb16}
{\sc B.~Armbruster}, {\em A simple and general proof for the convergence of
  markov processes to their mean-field limits}.
\newblock arXiv:1602.05224v2, 2016.

\bibitem{ArmBec16}
{\sc B.~Armbruster and E.~Beck}, {\em An elementary proof of convergence to the
  mean-field equations for an epidemic model}.
\newblock arXiv:1501.03250v4, 2016.

\bibitem{blah2}
{\sc F.~Bagagiolo}, {\em {Ordinary Differential Equations}}.
\newblock \url{http://www.science.unitn.it/~bagagiol/noteODE.pdf}, June 2016.
\newblock Proposition 6.4.

\bibitem{BaTeBa16}
{\sc D.~Bauso, H.~Tembine, and T.~Basar}, {\em Opinion dynamics in social
  networks through mean-field games}, SIAM Journal of Control and Optimization,
  54 (2016), pp.~3225--3257.

\bibitem{Bern1766-simple}
{\sc D.~Bernoulli}, {\em Essai d'une nouvelle analyse de la mortalit\'{e}
  caus\'{e}e par la petite v\'{e}role}, M\'{e}m.\ Math.\ Phys.\ Acad.\ Roy.\
  Sci., Paris,  (1766), pp.~1--45.

\bibitem{BoChGa+10}
{\sc C.~Borgs, J.~Chayes, A.~Ganesh, and A.~Saberi}, {\em How to distribute
  antidote to control epidemics}, Random Structures and Algorithms,  (2010).

\bibitem{blah1}
{\sc J.~V. Burke}, {\em {Ordinary Differential Equations}}.
\newblock
  \url{https://www.math.washington.edu/~burke/crs/555/555_notes/exist.pdf},
  June 2016.
\newblock page 15.

\bibitem{DieHee02}
{\sc K.~Dietz and J.~Heesterbeek}, {\em {Daniel Bernoulli's} epidemiological
  model revisited}, Mathematical Biosciences, 180 (2002), pp.~1--21.

\bibitem{EasKle10}
{\sc D.~Easley and J.~Kleinberg}, {\em Networks, Crowds, and Markets: Reasoning
  About a Highly Connected World}, Cambridge University Press, 2010.

\bibitem{EthKur86}
{\sc S.~N. Ethier and T.~G. Kurtz}, {\em Markov Processes: Characterization and
  Convergence}, John Wiley and Sons, 1986.

\bibitem{GanMasTow05}
{\sc A.~Ganesh, L.~Massouli\'{e}, and D.~Towsley}, {\em The effect of network
  topology on the spread of epidemics}, in Proceedings of IEEE INFOCOM '05,
  2005.

\bibitem{GaGaBo12}
{\sc N.~Gast, B.~Gaujal, and J.-Y.~L. Boudec}, {\em Mean field for markov
  decision processes: From discrete to continuous optimization}, IEEE
  Transactions on Automatic Control, 57 (2012), pp.~2266--2280.

\bibitem{Grob59}
{\sc D.~M. Grobman}, {\em Homeomorphism of systems of differential equations},
  Doklady Akad.\ Nauk SSSR, 128 (1959), pp.~880--881.

\bibitem{Grob62}
{\sc D.~M. Grobman}, {\em Topological classification of neighborhoods of a
  singularity in n-space}, Mat.\ Sbornik, 56 (1962), pp.~77--94.

\bibitem{guckenheimer:1983}
{\sc J.~Guckenheimer and P.~Holmes}, {\em Nonlinear Oscillations, Dynamical
  Systems, and Bifurcations of Vector Fields}, Springer-Verlag, 1983.

\bibitem{Hart60}
{\sc P.~Hartman}, {\em A lemma in the theory of structural stability of
  differential equations}, Proceedings of the American Mathematical Society, 11
  (1960), pp.~610--620.

\bibitem{Hart63}
{\sc P.~Hartman}, {\em On the local linearization of differential equations},
  Proceedings of the American Mathematical Society, 14 (1963), pp.~568--573.

\bibitem{hirsch:cooperative:1985}
{\sc M.~W. Hirsch}, {\em Systems of differential equations which are
  competitive and cooperative {II}: convergence almost everywhere}, SIAM
  Journal of Mathematical Analysis, 16 (1985), pp.~423--439.

\bibitem{Kurt71}
{\sc T.~G. Kurtz}, {\em Limit theorems for sequences of jump {Markov} processes
  approximating ordinary differential equations}, Journal of Applied
  Probability, 8 (1971), pp.~344--356.

\bibitem{LaSLef61}
{\sc J.~LaSalle and S.~Lefschetz}, {\em Stability by {Lyapunov's} Second Method
  with Applications}, Springer-Verlag, New York, 1961.

\bibitem{lorenz:1963}
{\sc E.~N. Lorenz}, {\em Deterministic nonperiodic flow}, Journal of the
  Atmospheric Sciences, 20 (1963), pp.~130--141.

\bibitem{LuSqWu+15}
{\sc Y.~Lu, M.~S. Squillante, C.~W. Wu, and B.~Zhang}, {\em On the control of
  epidemic-like stochastic processes with time-varying behavior}, in
  Proceedings of ACM SIGMETRICS Workshop on Mathematical Performance Modeling
  and Analysis, 2015.

\bibitem{Lyap1892}
{\sc A.~Lyapunov}, {\em The General Problem of the Stability of Motion (In
  Russian)}, PhD thesis, University of Kharkov, 1892.
\newblock English translation: The General Problem of the Stability of Motion,
  A.T. Fuller trans., Taylor \& Francis, London 1992.

\bibitem{NAS2013}
{\sc G.~H. Palmer et~al.}, {\em Using Science to Improve the Bureau of Land
  Management Wild Horse and Burro Program: A Way Forward}, National Academies
  Press, 2013.

\bibitem{peuteman:slowlyvarying:2002}
{\sc J.~Peuteman, J.~Peuteman, and D.~Aeyels}, {\em Exponential stability of
  slowly time-varying nonlinear systems}, Mathematics of Control, Signals and
  Systems, 15 (2002), pp.~202--228.

\bibitem{smith:mds2012}
{\sc H.~Smith}, {\em Dynamical systems in biology}.
\newblock \url{https://math.la.asu.edu/~halsmith/MDS-2012.pdf}, 2012.

\bibitem{Vidy02}
{\sc M.~Vidyasagar}, {\em Nonlinear Systems Analysis}, SIAM, 2002.

\bibitem{Whitt02}
{\sc W.~Whitt}, {\em Stochastic-Process Limits}, Springer-Verlag, New York,
  2002.

\end{thebibliography}
\end{document}